\documentclass[a4paper]{article}

\usepackage{amsthm}
\usepackage{amsmath}
\usepackage{amssymb}

\usepackage{paralist}

\usepackage{tikz}

\newtheorem{theorem}{Theorem}[section]
\newtheorem{lemma}[theorem]{Lemma}
\newtheorem{corollary}[theorem]{Corollary}

\newcommand{\N}{\mathbb{N}}
\newcommand{\R}{\mathbb{R}}
\newcommand{\Z}{\mathbb{Z}}

\allowdisplaybreaks

\begin{document}

\title{A camel with a less strict diet\\[5pt]
\large A variant of David Klarner's camel-banana problem}
\author{Michiel de Bondt}
\maketitle

\begin{abstract}
A camel can carry $B$ bananas on its back. It can have 
$2$ bananas at a time in its stomach. For each mile the camel 
walks, the amount of bananas in its stomach decreases $1$. As soon 
as the amount of bananas in the camel's stomach is at most $1$, 
it can eat a new banana. When the camel's stomach is empty, the
camel must eat a new banana (in order to be able to continue its
itinerary).

Let there be a stock of $N$ bananas at the border of the
desert. How far can the camel penetrate into the desert,
starting at this point? (Of course it can form new stocks
with transported bananas.)

The case $B=1$ is solved completely. The round trip variant is solved for $B=1$ as well. For $B=2$, the round trip variant is solved for $N$ which are a power of $2$ and $N \le 8$, and estimated up to $1/(N-1)$ miles for general $N$.
\end{abstract}

\section{Introduction}

In the original camel-banana problem by David Klarner, the camel had only room for $1$ banana in its stomach, and was able to carry at most $1$ banana. The original camel-banana problem appeared in a Dutch mathematical competition for students at the end of 1992. During that competition, the author formulated a strategy for all $N$, and proved its optimality for infinitely many values of $N$. Later in 1993, the author proved the optimality of the strategy for all $N$, which was published in \cite{cb}. At the same time, the camel-banana problem was solved in \cite{rote}.

The optimal strategy for the original camel-banana problem was called the UWC-strategy by the author, partly as a reference to `Ultra Wise Camel', but mainly as a reference to `Universitaire WiskundeCompetitie', which was the name of the mathematical competition for students. The partial solution of the original camel-banana problem during the competition and the complete solution of the original camel-banana problem after the competition were completely different. In this paper, we apply the techniques of the partial solution of the original camel-banana problem to solve the case $B=1$ of the above problem completely. That solution was already found by the author in 1996 or earlier. 

In 2024, the author considered the round trip variant of the above problem. We give a solution of the round trip variant of the case $B=1$. The round trip variant of both the original camel-banana problem and the above problem for $B=2$ seem hard to solve in general. To return to the desert border, the camel must lay down bananas on every mile, but what is the offset of the sequence of bananas? For $B = 2$, we will only provide a solution if $N = 2^k$ for some $k \in \N^{*}$ or $N \le 8$. In the first case, the camel can penetrate $k$ miles into the desert and return to the desert border. For $B=2$ and general $N$, we will estimate the round trip variant up to less than $1/(N-1)$ miles. But my guess is that the given strategy for the round trip variant of the case $B=2$ is optimal.

The case $B=2$, but not the round trip variant, can be solved completely by a reduction to the original camel-banana problem, which was found by the author long ago. This goes as follows. We can scale down down the dimensions by a factor $2$, while we preserve the rate of one banana per mile. The result of that is a camel which can eat at most one banana, but in portions of half a banana at a time, and can carry at most $\frac{B}2$ bananas. So the camel at hand has more flexibility in eating bananas than the camel of the original camel-banana problem. But one can prove that this flexibility does not lead to a farther destination in the one way trip variant if $B$ is even. We will not provide a proof of that in this paper. But in the near future, the author will share a generalization of this result.

The camel-banana problem is a discrete variant of the jeep problem. 
\begin{quotation} \small
A jeep can carry at most $F$ units of fuel at any time, and can travel one mile on one unit of fuel (the jeep's fuel consumption is assumed to be constant). At any point in a trip the jeep may leave any amount of fuel that it is carrying at a fuel dump, or may collect any amount of fuel that was left at a fuel dump on a previous trip, as long as its fuel load never exceeds $F$ units.

Let there be $N$ units of fuel stored at the border of the desert. How far can the jeep penetrate into the desert, starting at this point?
\end{quotation}
The jeep problem was solved in \cite{jeep}. The jeep can get as far as
$$
\frac{F}1 + \frac{F}3 + \frac{F}5 + \cdots + \frac{F}{2\lfloor N/F \rfloor - 1} + \frac{N - F\lfloor N/F \rfloor}{2\lfloor N/F \rfloor + 1}
$$
miles into the desert in the one way trip case, and
$$
\frac{F}2 + \frac{F}4 + \frac{F}6 + \cdots + \frac{F}{2\lfloor N/F \rfloor} + \frac{N - F\lfloor N/F \rfloor}{2\lfloor N/F \rfloor + 2}
$$
miles into the desert in the round trip case, where $N$ is any positive real number. A jeep which can carry $F = B+2$ units of fuel can get at least as far as a camel which can carry $B$ bananas on its back and $2$ bananas in its stomach, since it has more options with the fuel it carries. But one can show that in the one way trip case, the following holds: if $N  \le (B+2) \big\lceil \frac{B+2}2 \big\rceil$, then the camel can reach as far into the desert as a jeep which can carry $F = B+2$ units of fuel.

Lower bounds of the destination can be given by way of strategy recipes for the camel. We now describe the key of the upper bounds. Define $p_B: \R \rightarrow \R$ by 
$$
p_B(x) = \big(1 + \tfrac2B\big)^{\lfloor x \rfloor}\big(1 + \tfrac2B(x - \lfloor x \rfloor)\big)
$$
If $x \in \Z$, then $p_B(x)$ is just $\big(1 + \tfrac2B\big)^{x}$. For other values of $x$, $p_B(x)$ is obtained by linear interpolation of $p_B(\lfloor x \rfloor)$ and $p_B(\lceil x \rceil)$.

If the amount of banana fuel in the camel's stomach is an integer, then this amount is $0$, $1$, or $2$. If this amount is $0$, then the camel is about to eat a banana. If this amount is $2$, then the camel has just eaten a banana. So with intervals of one mile, the camel has exactly the fuel of one banana in its stomach. At such moments, the following potential is computed:
$$
P_B := \sum_{\text{$b$ is a remaining banana}} p_B(\text{position of $b$}) - \tfrac{B}2 p_B(\text{position of camel})
$$
The first time $P_B$ is computed, there are $N - 1$ bananas left, all on position $0$. Therefore $P_B = N-1-\frac{B}2$ the first time. 

\begin{lemma}
$P_B$ cannot increase.
\end{lemma}

\begin{proof}
Between two consecutive moments that $P_B$ is computed, the camel walks a mile and eats a banana. Suppose that the camel eats this banana on position $x$. Then the costs for $P_B$ of eating that banana are $p_B(x)$. But during the walk of one mile, the camel can gain for $P_B$. 

Suppose that the camel walks on a position $y \notin \N$. If the camel walks in forward direction with $B$ bananas on its back, then the gain for $P_B$ on that moment is $-\frac{B}2 p'(y)$ per mile for the camel and $B p'(y)$ per mile for the bananas, which adds up to $\frac{B}2 p'(y)$ per mile. If the camel walks in backward direction without bananas on its back, then the gain for $P_B$ on that moment is $\frac{B}2 p'(y)$ per mile, just as in the forward case.

Notice that $\frac{B}2 p'(y) = \big(1 + \tfrac2B\big)^{\lfloor y \rfloor}$ is increasing. So a maximal gain for $P_B$ is obtained if the camel walks from position $x$ to position $x + 1$ with $B$ bananas on its back (or from position $x+1$ to position $x$ without bananas on its back), namely
$$
\int_{x}^{x+1} \tfrac{B}2 p'(y) dy = \tfrac{B}2 \big(p(x+1) - p(x)\big) = \tfrac{B}2 \big(\big(1+\tfrac2B\big) p(x) - p(x)\big) = p(x)
$$
This gain compensates the costs of eating a banana on position $x$ exactly.
\end{proof}

To make that $P_B$ does not decrease during the mile at hand, the camel does not need to make a walk between $x$ and $x+1$ in all cases. If $y < \lceil x \rceil$, then $p'(y)$ is small. If $y > \lceil x \rceil$, then $p'(y)$ is large. Besides carrying $B$ or $0$ bananas depending on the camel's direction, it suffices to maximize the fraction of the mile at hand with large $p'(y)$. If $x = \lceil x \rceil$, then the camel can do that by walking $1$ mile on positions $\ge \lceil x \rceil$. If $x < \lceil x \rceil$, then the camel can do that in two ways. Either it first walks from position $x$ to position $\lceil x \rceil$ and next it walks $x+1 - \lceil x \rceil$ miles on positions $\ge \lceil x \rceil$. Or it first walks $x+1 - \lceil x \rceil$ miles on positions $\ge \lceil x \rceil$ and next it walks from position $\lceil x \rceil$ to position $x$. This insight is useful for finding good strategies for the camel.

We assumed that the number of bananas $N$ is an integer. But there is also an interpretation for arbitrary positive $N$, namely as follows. The camel starts with $\lfloor N \rfloor$ bananas next to him, and $N - \lfloor N \rfloor$ banana units in its stomach. Now $P_B$ is computed for the first time after $N - \lfloor N \rfloor$ miles. Before that, the camel can eat a banana on position $0$ and carry $B$ bananas to position $N - \lfloor N \rfloor$. After that, there are $\lfloor N \rfloor - B - 1$ bananas on position $0$ and $B$ bananas on position $N - \lfloor N \rfloor$ with the camel, so
$$
P_B = \lfloor N \rfloor - B - 1 + \big(B-\tfrac{B}2\big)\big(1 + \tfrac2B(N - \lfloor N \rfloor)\big) = N - 1 - \tfrac{B}2
$$
$P_B$ cannot be larger, because the costs of eating a banana is minimal on position $0$, and the gain of walking on any position $y \in (0,1)$ is at most $\frac{B}2 p'(y) = 1$ per mile. Furthermore a gain of $1$ per mile is attained in $(0,1)$ if the camel carries $B$ bananas when it walks in forward direction and $0$ bananas when it walks in backward direction.

\section{Upper bounds}

\begin{theorem} \label{onewaytrip}
In the one way trip variant, the camel cannot penetrate farther than
$$
p_B^{-1}\bigg(\frac{2N-2-B}B\bigg) + B + 1
$$
miles into the desert if $N \ge B+1$.
\end{theorem}

\begin{proof}
Suppose that the camel is on position $x$ when $P_B$ is computed with $B$ bananas left. Say that these $B$ bananas are on positions $\ge y$, including $y$ itself. Suppose first that $x \le y$. Then
$$
N - 1 - \tfrac{B}2 \ge P_B \ge Bp_B(y) - \tfrac{B}2 p_B(x) \ge \tfrac{B}2 p_B(x)
$$
The camel can get at most as far as $x + B + 1 \le p_B^{-1}\big(\frac{2N-2-B}B\big) + B+1$.

Suppose next that $x > y$. We may assume that all remaining bananas are used, so the camel must walk from position $x$ back to position $y$ to get all remaining bananas. After that, it does not have enough fuel to reach $y+B+1$. But the camel would have reached $y+B+1$ if $x$ had been $y$. To accomplish that, the camel can refuse to go farther than $y$ before the moment that $P_B$ is computed with $B$ remaining bananas. So instead of being on position $v$, the camel is on position $\min\{v,y\}$. The same holds for the bananas.
\end{proof}

\begin{theorem} \label{roundtrip}
In the round trip variant, the camel cannot penetrate farther than
$$
p_1^{-1}\bigg(\frac{N}{B}\bigg) + \frac{B}{2}
$$
miles into the desert if $N \ge B$.
\end{theorem}

\begin{proof}
If $B \le N \le B+2$, then the camel can penetrate exactly $\frac{N}{2}$ miles into the desert with returning, and
$$
p_B^{-1}\big(\tfrac{N}{B}\big) + \tfrac{B}2 = p_B^{-1}\big(1 + \tfrac{2}{B}\tfrac{N-B}2\big) + \tfrac{B}2 = \tfrac{N-B}2 + \tfrac{B}2 = \tfrac{N}2
$$
If $N \ge B+2$, then 
$$
p_B^{-1}\big(\tfrac{N}{B}\big) + \tfrac{B}2 \ge p_B^{-1}\big(1 + \tfrac{2}{B}\big) + \tfrac{B}2 = \tfrac{B}2 + 1
$$

Let $c$ be the farthest position which the camel reaches, say on moment $t$, and assume that $c \ge \frac{B}{2} + 1$. The camel starts at least $\frac{B}{2} + 1$ miles before moment $t$, and $P_B$ is computed on a position $x$ $m_x$ miles before moment $t$, for some $m_x \in \big(\frac{B}2,\frac{B}2+1\big]$. The camel ends at least $\frac{B}{2} + 1$ miles after moment $t$, and $P_B$ is computed on a position $z$ $m_z$ miles after moment $t$, for some $m_z \in \big[\frac{B}2-\frac12,\frac{B}2+\frac12\big)$. Let $M$ be the miles between the computations of $P_B$ on positions $x$ and $z$. We distinguish two cases.
\begin{compactitem}

\item $m_z < \frac{B}2$.

Then $B - \frac12 < m_x + m_z < B + 1$, so $m_x + m_z = B$. Furthermore, $m_z = c - z$ and
$$
m_x = B - m_z = m_z + B - 2m_z \le m_z + B - (B-1) = c - (z-1)
$$
Consequently, the camel cannot reach positions $< z-1$ in $M$. So in $M$, the camel cannot transport the bananas which are eaten on the positions $z-1, z-2, \ldots$ during the return trip. If we assume that the camel has one banana in its stomach at the beginning and the end of $M$, then the camel eats $B$ bananas in $M$. Say that at the beginning of $M$, these $B$ bananas are on positions $\ge y$, including $y$ itself.

Suppose first that $x \le \min\{y,z\}$. Then the camel can get as far as $\frac12 x + \frac12 z + \frac{B}2$, namely by walking from $x$ straight to $\frac12 x + \frac12 z + \frac{B}2$ and next straight to $z$. The camel cannot get farther, so $c \le \frac12 x + \frac12 z + \frac{B}2$. The decrement of $P_B$ in $M$ is at least
\begin{align*}
B p_B(y)  - \tfrac{B}2 p_B(x) + \tfrac{B}2 p_B(z)
&\ge B p_B(x) - \tfrac{B}2 p_B(x) + \tfrac{B}2 p_B(z) \\
&= B\big(\tfrac12 p_B(x) + \tfrac12 p_B(z)\big) \\
&\ge Bp_B\big(\tfrac12 x + \tfrac12 z\big) \\
&\ge Bp_B\big(c-\tfrac{B}2\big)
\end{align*}

Suppose next that $x > \min\{y,z\}$. As $c - z = m_z < \frac{B}2$, $c < \frac12 z + \frac12 z + \frac{B}2$ follows. If $y < z$, then the camel must fetch the banana on position $y$ before reaching $c$, so $c < \frac12 y + \frac12 z + \frac{B}2$. So $c < \frac12 \min\{y,z\} + \frac12 z + \frac{B}2$. But the camel can reach $\frac12 \min\{y,z\} + \frac12 z + \frac{B}2$. To accomplish that, the camel can refuse to go farther than $\min\{y,z\}$ before $M$. So instead of being on position $v$, the camel is on position $\min\{v,y,z\}$. The same holds for the bananas. In $M$, the camel walks from position $\min\{y,z\}$ straight to position $\frac12 \min\{y,z\} + \frac12 z + \frac{B}2$, and next straight to position $z$.

\item $m_z \ge \frac{B}2$.

Then $B < m_x + m_z < B + 1\frac12$, so $m_x + m_z = B+1$. Furthermore, $m_z = c - z$ and
$$
m_x = (B + 1) - m_z = m_z + (B + 1) - 2m_z \le m_z + (B + 1) - B = c - (z-1)
$$
Consequently, the camel cannot reach positions $< z-1$ in $M$. So in $M$, the camel cannot transport the bananas which are eaten on the positions $z-1, z-2, \ldots$ during the return trip. If we assume that the camel has one banana in its stomach at the beginning and the end of $M$, then the camel eats $B+1$ bananas in $M$. Say that at the beginning of $M$, these $B+1$ bananas are on positions $\ge y$, including $y$ itself.

Suppose first that $x \le \min\{y,z\}$. Then the camel can get as far as $\frac12 x + \frac12 z + \frac{B+1}2$, namely by walking from $x$ straight to $\frac12 x + \frac12 z + \frac{B+1}2$ and next straight to $z$. The camel cannot get farther, so $c \le \frac12 x + \frac12 z + \frac{B+1}2$. The decrement of $P_B$ in $M$ is at least
\begin{align*}
(B+1)p_B(y) - \tfrac{B}2 p_B(x) + \tfrac{B}2 p_B(z)
&\ge \tfrac{B}2\big(1+\tfrac2{B}\big)p_B(x) + \tfrac{B}2 p_B(z) \\
&= B\big(\tfrac12 p_B(x+1) + \tfrac12 p_B(z)\big) \\
&\ge Bp_B\big(\tfrac12 x + \tfrac12 + \tfrac12 z\big) \\
&\ge Bp_B\big(c-\tfrac{B}2\big)
\end{align*}

Suppose next that $x > \min\{y,z\}$. As $c - z = m_z < \frac{B}2 + \frac12$, $c < \frac12 z + \frac12 z + \frac{B+1}2$ follows. If $y < z$, then the camel must fetch the banana on position $y$ before reaching $c$, so $c < \frac12 y + \frac12 z + \frac{B+1}2$. So $c < \frac12 \min\{y,z\} + \frac12 z + \frac{B+1}2$. But the camel can reach $\frac12 \min\{y,z\} + \frac12 z + \frac{B+1}2$. To accomplish that, the camel can refuse to go farther than $\min\{y,z\}$ before $M$. So instead of being on position $v$, the camel is on position $\min\{v,y,z\}$. The same holds for the bananas. In $M$, the camel walks from position $\min\{y,z\}$ straight to position $\frac12 y + \frac12 z + \frac{B+1}2$, and next straight to position $z$.

\end{compactitem}
The first time $P_B$ is computed, we have $P_B \le N - 1 - \frac{B}2$. The last time $P_1$ is computed, the camel is on a position $\le 1$, so $P_B \ge -\frac{B}2\big(1+\frac2B\big) = -1 - \frac{B}2$. So the decrement of $P_B$ is at most $N$. Hence $Bp_B\big(c-\frac{B}2\big) \le N$, so $c \le p_B^{-1}\big(\frac{N}{B}\big) + \frac{B}2$.
\end{proof}

If $B = 2$, then the upper bound in theorem \ref{roundtrip} cannot be attained if $4 < N < 8$. We use jeep reasoning to obtain an upper bound which can be attained. Before eating its last banana on a position $z \in (0,1]$, we allow the camel to act as a jeep which can carry $B+2 = 4$ units of fuel.  

\begin{theorem} \label{jeepz}
In the round trip variant of the case $B=2$, the camel can penetrate at most
$$
\frac{N-4}6 + 2
$$
miles into the desert if $4 \le N \le 6$, and at most
$$
\frac{N-5}3 + 2
$$
miles into the desert if $6 \le N \le 8$.
\end{theorem}

\begin{proof}
Suppose that $4 \le N \le 8$. We must show that the camel can penetrate at most $\max\big\{\frac{N-4}6 + 2,\frac{N-5}3 + 2\big\}$ miles into the desert with returning if $B = 2$. Let $z$ be the position where the camel eats the last banana. Let $c$ be the farthest position the camel reaches. Then $c \ge 2$. We distinguish $3$ cases.
\begin{compactitem}

\item $z \le c-2$.

Let $0 < v < c-2$. Before eating its last banana on position $z$, the camel must walk from position $v$ to position $c$ and from position $c$ to position $c - 2$, which adds up to $c - v + 2 > 4$ miles. The fuel for these miles comes from position $0$. Since only $4$ units of fuel can be transported at a time by the camel (or a jeep which can carry $4$ units of fuel), the camel must walk at least $2(v-0)$ miles in forward direction between position $0$ and position $v$. So the camel must walk at least $v-0$ miles in backward direction between position $0$ and position $v$. The camel must walk at least $2\big(c-(c-2)\big)$ miles between position $c-2$ and position $c$. Taking the limit of $v \rightarrow c-2$, we infer that
$$
3c - 2 = 2c + (c-2) = 3\big((c-2) - 0\big) + 2\big(c-(c-2)\big) \le N-1
$$
Hence $c \le \frac{N+1}3 = \frac{N-5}3 + 2$.

\item $c-2 \le z \le c-1\frac12$.

Let $0 < v < z$. Before eating its last banana on position $z$, the camel must walk from $v$ to $c$ and from $c$ to $z$, which adds up to $(c - v) + (c - z) > 2(c - z) \ge 3$ miles. The fuel for these miles and for the banana on position $z$ comes from position $0$. Since only $4$ units of fuel can be transported at a time by the camel (or a jeep which can carry $4$ units of fuel), the camel must walk at least $2(v-0)$ miles in forward direction between position $0$ and position $v$. So the camel must walk at least $v-0$ miles in backward direction between position $0$ and position $v$. The camel must walk at least $2(c-z)$ miles between position $z$ and position $c$. Taking the limit of $v \rightarrow z$, we infer that
$$
3c - 2 = 2c + (c-2) \le 2c + z = 3(z-0) + 2(c-z) \le N-1
$$
Hence $c \le \frac{N+1}3 = \frac{N-5}3 + 2$.

\item $c-1\frac12 \le z \le 1$.

Let $w = 2c - z - 3$. Then
$$
c-2 \le 2(c-2) = 2c - 1 - 3 \le w \le 2c - \big(c + 1\tfrac12\big) - 3 = c - 1\tfrac12 \le z
$$
Let $0 < v < w$. Before eating its last banana on position $z$, the camel must walk from $v$ to $c$ and from $c$ to $z$, which adds up to $(c - v) + (c - z) = 2c - z - v > 3$ miles. The fuel for these miles and for the banana on position $z$ comes from position $0$. Since only $4$ units of fuel can be transported at a time by the camel (or a jeep which can carry $4$ units of fuel), the camel must walk at least $2(v-0)$ miles in forward direction between position $0$ and position $v$. So the camel must walk at least $v-0$ miles in backward direction between position $0$ and position $v$. The camel must walk at least $z-w$ miles between position $w$ and position $z$, and at least $2(c-z)$ miles between position $z$ and position $c$. Taking the limit of $v \rightarrow w$, we infer that
\begin{align*}
6c - 9 &\le 6c - 3z - 6 = 3w + 3 \\ 
&= 2c - z + 2w = 3(w-0) + (z-w) + 2(c-z) \le N-1
\end{align*}
Hence $c \le \frac{N+8}6 = \frac{N-4}6 + 2$.

\end{compactitem}
\end{proof}

\begin{theorem}
In the original camel-banana problem, the camel cannot penetrate farther than
$$
\tfrac12 p_2^{-1}\big(\tfrac23(N-1)\big) + 2\tfrac16
$$
miles into the desert if $N \ge 3$.
\end{theorem}

\begin{proof}
If we scale up the situation of the original camel-banana problem by a factor $2$, then we get a camel which can have $2$ bananas at a time in its stomach and can carry $2$ bananas, but the camel eats its bananas in pairs. Suppose that the camel is on position $x$ when $P_2$ is computed with $5$ remaining bananas. Suppose that the smallest position with a banana is $y$ and the largest position with a banana is $z$. Suppose that the farthest two bananas lie on positions $z'$ and $z$.

We first show that the camel cannot get farther than $\frac13 y + \frac23 z + 4 \frac23$. The first case is $z \ge y + 2$. Then $\frac13 y + \frac23 z + 4\frac23 \ge y + 6$. It is clear that with $6$ units of banana fuel, the camel cannot get farther than position $y + 6$.

The second case is $z < y + 2$. Let $c$ be the farthest position. If $c \le z + 4$, then  
$c < \frac13 y + \frac23 z + 4\frac23$. So assume that $c > z + 4$. Take $v$ such that $z < v < c-4$. Walking from position $v$ to position $c$ takes $c  - v > 4$ units of fuel. This fuel comes from position $z$ and smaller positions. Since only $4$ units of fuel can be transported at a time by the camel (or a jeep which can carry $4$ units of fuel), the camel must walk at least $2(v-z)$ miles in forward direction between position $z$ and position $v$. So the camel must walk at least $v-z$ miles in backward direction between position $z$ and position $v$. The camel must walk at least $z-y$ miles between position $y$ and position $z$, and at least $c-(c-4)$ miles between position $c-4$ and position $c$. Taking the limit of $v \rightarrow c-4$, we infer that
$$
(z - y) + 3\big((c-4)-z\big) + \big(c - (c-4)\big) \le 6
$$
Hence $3c - 2z - y - 8 \le 6$ and $c \le \frac13 y + \frac23 z + 4\frac23$.

If $z' < z$, then every position between $z'$ and $z$ is crossed by an odd number of bananas. Therefore $P_2$ has decreased $p_2(z) - p_2(z')$ from the beginning. Since the camel eats the bananas in pairs, the camel immediately eats a banana on position $x$ after $P_2$ is computed with $5$ remaining bananas. So
$$
N - 2 - \big(p_2(z) - p_2(z')\big) \ge P_2 \ge p_2(z) + p_2(z') + 2 p_2(y) + p_2(x) - p_2(x)
$$
Therefore
\begin{align*}
N - 2 &\ge 2 p_2(z) + 2 p_2(y)  = 2 p_2(z) + p_2(y+1) = 3 \big(\tfrac13 p_2(y+1) + \tfrac23 p_2(z)\big) \\
&\ge 3 p_2\big(\tfrac13 y + \tfrac13 + \tfrac23 z\big) \ge 3 p_2\big(c - 4 \tfrac13\big)
\end{align*}
and $c \le p_2^{-1}\big(\frac13 (N - 2)\big) + 4\frac13$. Scaling down the dimensions by a factor $2$ gives $2c \le p_2^{-1}\big(\frac13 (2N - 2)\big) + 4\frac13$, which is equivalent to $c \le \frac12 p_2^{-1}\big(\frac23 (N - 1)\big) + 2\frac16$.
\end{proof}

\begin{corollary} \label{cor}
Let $k \in \N^{*}$ and $0 \le f \le 2$. With $2^k + f$ bananas, the camel of the original camel-banana problem cannot penetrate farther than
$$
\tfrac12 k + 1\tfrac56 + \frac{f-1}{3 \cdot 2^{k-1}}
$$
miles into the desert.
\end{corollary}

\begin{proof}
Let $c$ be the distance the camel can penetrate into the desert with $N = 2^k + f$ bananas. Suppose first that $k = 1$. Then $2 \le N \le 4$. Take $v$ such that $0 < v < c-2$. Walking from position $v$ to position $c$ takes $c  - v > 2$ units of fuel. This fuel comes from position $0$. Since only $2$ units of fuel can be transported at a time by the camel (or a jeep which can carry $2$ units of fuel), the camel must walk at least $2(v-0)$ miles in forward direction between position $0$ and position $v$. So the camel must walk at least $v-0$ miles in backward direction between position $0$ and position $v$. The camel must walk at least $c - (c-2)$ miles between position $c-2$ and position $c$. Taking the limit of $v \rightarrow c-2$, we infer that
$$
3\big((c-2)-0\big) + \big(c-(c-2)\big) \le N
$$
Hence $3c - 6 + 2 \le N$ and 
$$
c \le \tfrac13 (N+4) = \tfrac43 + \tfrac13 (2 + f) = \tfrac12 + \tfrac56 + 1 + \tfrac13(f-1) = \tfrac12 k + 1\tfrac56 + \frac{f-1}{3 \cdot 2^{k-1}}
$$

Suppose next that $k \ge 2$. Then $2^k \le \frac43(2^k-1) \le \frac43(N-1) \le \frac43(2^k+1) < 2^{k+1}$. So
\begin{align*}
\tfrac12 p_2^{-1}\big(\tfrac23(N-1)\big) + 2\tfrac16 
&= \tfrac12 k - \tfrac12 + \tfrac12 p_2^{-1}\bigg(\frac{\frac43(N-1)}{2^k}\bigg) + 2\tfrac16 \\
&= \tfrac12 k + 1\tfrac23 + \tfrac12 \bigg(\frac{\frac43(N-1)}{2^k} - 1\bigg) \\
&= \tfrac12 k + 1\tfrac23 + \tfrac12 \frac{N-1-3 \cdot 2^{k-2}}{3 \cdot 2^{k-2}} \\
&= \tfrac12 k + 1\tfrac23 + \frac{f-1 + 4 \cdot 2^{k-2} - 3 \cdot 2^{k-2}}{3 \cdot 2^{k-1}} \\
&= \tfrac12 k + 1\tfrac56 + \frac{f-1}{3 \cdot 2^{k-1}}
\end{align*}
which completes the proof.
\end{proof}

\section{Lower bounds}

\begin{lemma}
Let $B = 1$ and $n \in \N^{*}$. With $3n$ bananas, the camel can penetrate (at least) one mile farther into the desert than with $n+1$ bananas. If the objective is to deliver one of the bananas and return to the desert border, then with $3n$ bananas, the camel can deliver the banana (at least) one mile farther than with $n$ bananas.

Let $N \in \R$ such that $n \le N \le 3n$. With $N$ bananas, the camel can penetrate (at least) $(N - n - 1)/(2n - 1)$ miles farther than with $n + 1$ bananas.
\end{lemma}

\begin{proof}
Suppose that the camel is at the desert border with $3n$ bananas. It carries $n$ bananas to position $1$, consuming $2$ bananas for each of the $n$ walks to position $1$ and back, making $2n$ consumed bananas altogether. But after carrying the last banana to position $1$, it does not need to go back, so the camel saves one banana in its stomach. So the camel ends on position $1$ with $n+1$ bananas, of which one banana is in its stomach. If the objective is to deliver one of the bananas and return to the desert border, then the camel must save one banana to return to the desert border at the end, so only $n$ bananas are available for the delivery with return from position $1$.

Suppose that the camel is at the desert border with $N$ bananas. It carries $n$ bananas to position $x = (N - n - 1)/(2n - 1)$, using $N - n$ bananas for walking and ending with one banana in its stomach. Indeed, it walks $(2n-1)x$ miles, so it needs to consume $(2n-1)x + 1 =  N - n$ bananas. But these $N - n$ bananas must be eaten on the desert border. This can be done as follows. We have the camel eat two of the $N-n$ bananas at a time, except maybe the first time. So the camel eats $\lfloor (N-n) / 2 \rfloor$ times a pair of bananas.

To each of these $\lfloor (N-n) / 2 \rfloor$ meals, we assign a banana which will be carried from the desert border to position $x$ after that. Then the camel has $2 - 2x$ or $1 - x$ miles left, depending on if it returns to the desert border after the transport. The camel uses those remaining miles to carry unassigned bananas farther. It does not matter how far these bananas have been carried thus far, since the camel meets every position in $[0,x]$. With the fuel of the first meal, the camel only carries unassigned bananas if that meal consists of less than two bananas.
\end{proof}

\begin{theorem}
In the one way trip variant of the case $B=1$, the camel can penetrate exactly
$$
p_1^{-1}(2N-3) + 2
$$
miles into the desert if $N \ge 2$.
\end{theorem}

\begin{proof}
Suppose that $N \ge 2$. From theorem \ref{onewaytrip}, it follows that the camel cannot penetrate farther into the desert. Take $k \in \N$ such that $\frac12 3^k + 1\frac12 \le N \le \frac12 3^{k+1} + 1\frac12$. Let $n = \frac12 3^k + \frac12$. With $\frac12 3^0 + 1\frac12 = 2$ bananas, the camel can penetrate $2$ miles into the desert. With $\frac12 3^{k+1} + 1\frac12 = 3n$ bananas, the camel can penetrate $1$ mile farther into the desert than with $n + 1 = \frac12 3^k + 1\frac12$ bananas. By induction, it follows that the camel can penetrate $k+2$ miles into the desert with $n + 1$ bananas.

With $N$ bananas, the camel can penetrate 
$$
\frac{N-n-1}{2n-1} + k + 2
= p_1^{-1}\bigg(p_1\bigg(\frac{N-\frac12 3^k - 1\frac12}{3^k}\bigg)\bigg) + k + 2
$$
miles into the desert. Since $\frac12 3^k + 1\frac12 \le N \le \frac12 3^{k+1} + 1\frac12$, it follows that 
$$
0 \le \frac{N-\frac12 3^k - 1\frac12}{3^k} \le \frac{\frac12 3^{k+1} - \frac12 3^k}{3^k} = 1
$$
so 
$$
p_1\bigg(\frac{N - \frac12 3^k - 1\frac12}{3^k}\bigg) 
= 2\cdot\frac{N - \frac12 3^k - 1\frac12}{3^k} + 1 = \frac{2N-3}{3^k}
$$
Therefore the camel can penetrate $p_1^{-1}\big((2N-3)/3^k\big) + k + 2 = p_1^{-1}(2N-3) + 2$ miles into the desert.
\end{proof}

\begin{theorem}
In the round trip variant of the case $B=1$, the camel can penetrate exactly
$$
p_1^{-1}(N) + \tfrac12
$$
miles into the desert if $N \ge 1$.
\end{theorem}

\begin{proof}
Suppose that $N \ge 1$. From theorem \ref{roundtrip}, it follows that the camel cannot penetrate farther that $p_1^{-1}(N) + \tfrac12$ miles into the desert with returning. Furthermore, the proof of theorem \ref{roundtrip} shows that $p_1^{-1}(N) + \tfrac12$ miles can be attained if $1 \le N \le 3$.

Suppose that $N \ge 2$, and take $k \in \N$ such that $2 \cdot 3^k \le N \le 2 \cdot 3^{k+1}$. Delivering a banana on position $j+1$ with returning can be done with three times as many bananas as delivering a banana on position $j$ with returning. By induction, it follows that delivering a banana on position $j$ with returning can be done with $3^j$ bananas. We first have the camel deliver bananas on positions $1$, $2$, $\ldots$, $k$, and return to the desert border afterwards. This takes $n = (3^{k+1}-3)/(3-1) = 1\frac12 \cdot 3^k - 1\frac12$ bananas.

With the other $N - n$ bananas, the camel can penetrate $p_1^{-1}\big(2(N-n)-3) + 2$ miles into the desert with a one way trip. But it must now bend its path back to the banana on position $k$, so it can penetrate 
\begin{align*}
\tfrac12 \big(p_1^{-1}(2(N-n)-3) + 2\big) + \tfrac12 k 
&= \tfrac12 p_1^{-1}(2N - 3^{k+1}) + 1 - \tfrac12 + \tfrac12(k+1) \\
&= \tfrac12 p_1^{-1}(2N - 3^{k+1}) + \tfrac12 p^{-1}(3^{k+1}) + \tfrac12
\end{align*}
miles into the desert. From $2 \cdot 3^k \le N \le 2 \cdot 3^{k+1}$, it follows that
$$
3^k = 4 \cdot 3^k - 3^{k+1} \le 2N - 3^{k+1} \le 4 \cdot 3^{k+1} - 3^{k+1} = 3^{k+2} 
$$
Consequently, $p_1^{-1}$ is linear between $2N - 3^{k+1}$ and $3^{k+1}$. So the camel can penetrate 
\begin{align*}
\tfrac12 p_1^{-1}(2N - 3^{k+1}) + \tfrac12 p^{-1}(3^{k+1}) + \tfrac12 
&= p_1^{-1}\big(N - \tfrac12 \cdot 3^{k+1} + \tfrac12 \cdot 3^{k+1}\big) + \tfrac12 \\
&= p_1^{-1}(N) + \tfrac12
\end{align*}
miles into the desert.
\end{proof}

\begin{lemma}
Let $B = 2$ and $n \in \N^{*}$. With $4n$ bananas, the camel can penetrate (at least) one mile farther into the desert with returning than with $2n$ bananas. 

Let $N \in \R$ such that $2n \le N \le 4n$. If $N \le 2n+2$, then with $N$ bananas, the camel can penetrate (at least) $(N-2n)/(4n-2)$ miles farther into the desert with returning than with $2n$ bananas. If $N \ge 2n+1$, then with $N$ bananas, the camel can penetrate (at least) $(N-1-2n)/(2n-1)$ miles farther into the desert with returning than with $2n$ bananas. 
\end{lemma}

\begin{proof}
Suppose that the camel is at the desert border with $4n$ bananas. It carries $2n$ bananas to position $1$, consuming $2$ bananas for each of the $n$ walks to position $1$ and back, making $2n$ consumed bananas altogether. But after carrying the last banana to position $1$, it does not need to go back, so the camel saves one banana in its stomach. So the camel ends on position $1$ with $n+1$ bananas, of which one banana is in its stomach. The camel must save one banana to return to the desert border at the end, so only $n$ bananas are available for the delivery with return from position $1$.

Suppose that the camel is at the desert border with $N \le 2n+2$ bananas. It carries $2n$ bananas to position $x = (N - 2n)/(2n - 1)$, using $N - 2n$ bananas for walking and ending with an empty stomach. Indeed, it walks $(2n-1)x$ miles, so it needs to consume $(2n-1)x = N - 2n$ bananas. As $N-2n \le 2$, this can be done immediately at the desert border. From position $x$, the camel can use the $2n$ remaining bananas to penetrate into the desert and return to position $x$. But the camel must return to position $0$. This can be accomplished by returning $x/2$ miles earlier. So the camel can get $x/2 = (N - 2n)/(4n - 2)$ miles farther with return than with $2n$ bananas.

Suppose that the camel is at the desert border with $N \ge 2n+1$ bananas. It carries $2n$ bananas to position $x = (N - 1 - 2n)/(2n - 1)$, using $N - 2n$ bananas for walking and ending with one banana in its stomach. Indeed, it walks $(2n-1)x$ miles, so it needs to consume $(2n-1)x + 1 = N - 2n$ bananas. But these $N - 2n$ bananas must be eaten on the desert border. This can be done as follows. We have the camel eat two of the $N-2n$ bananas at a time, except maybe the first time. So the camel eats $\lfloor (N-2n) / 2 \rfloor$ times a pair of bananas.

To each of these $\lfloor (N-2n) / 2 \rfloor$ meals, we assign two banana which will be carried from the desert border to position $x$ after that. Then the camel has $2 - 2x$ or $1 - x$ miles left, depending on if it returns to the desert border after the transport. The camel uses those remaining miles to carry unassigned bananas farther in pairs. It does not matter how far these bananas have been carried thus far, since the camel meets every position in $[0,x]$. With the fuel of the first meal, the camel only carries unassigned bananas if that meal consists of less than two bananas.
\end{proof}

\begin{theorem} \label{roundtrip2}
In the round trip variant of the case $B=2$, the camel can penetrate at most
$$
p_2^{-1}(N)
$$
miles and more than 
$$
p_2^{-1}(N) - \frac1{N-1}
$$
miles into the desert if $N \ge 2$. The upper bound $p_2^{-1}(N)$ can be attained if either $2 \le N \le 4$ or $N = 2^k$ for some $k \in \N^{*}$.

If $4 < N < 8$, then the upper bound $p_2^{-1}(N) = \frac{N-4}4 + 2$ cannot be attained. In the round trip variant of the case $B=2$, the camel can penetrate exactly
$$
\frac{N-4}6 + 2
$$
miles into the desert if $4 \le N \le 6$, and exactly
$$
\frac{N-5}3 + 2
$$
miles into the desert if $6 \le N \le 8$.
\end{theorem}

\begin{proof}
Suppose that $N \ge 2$. From theorem \ref{roundtrip}, it follows that the camel cannot penetrate farther that $p_2^{-1}\big(\frac{N}{2}) + 1 = p_2^{-1}(N)$ miles into the desert with returning. Furthermore, the proof of theorem \ref{roundtrip} shows that $p_1^{-1}(N)$ miles can be attained if $2 \le N \le 4$.

Take $k \in \N^{*}$ such that $2^k \le N \le 2^{k+1}$, and let $n = 2^{k-1}$. With $2^1 = 2$ bananas, the camel can penetrate $1$ mile into the desert and return. With $2^{k+1} = 4n$ bananas, the camel can penetrate $1$ mile farther into the desert with returning than with $2n = 2^k$ bananas. By induction, it follows that the camel can penetrate $k$ miles into the desert with returning with $2n$ bananas. So if $N = 2^k$, then the camel can penetrate $k = p_2^{-1}(N)$ miles into the desert with returning.

Suppose first that $N \le 2n + 2$. With $N$ bananas, the camel can penetrate
\begin{align*}
\frac{N-2n}{4n-2} + k &= \frac{N-2^k}{2(2^k-1)} + k \\
&= -\frac{(N-2^k)(2(2^k-1)-2^k)}{2(2^k-1)2^k} + p_2^{-1}\bigg(p_2\bigg(\frac{N-2^k}{2^k}\bigg)\bigg) + k \\
&> -\frac{(N-2^k)(2^k-2)}{2(2^k-2)(2^k+1)} + p_2^{-1}\bigg(\frac{N}{2^k}\bigg) + k \\
&= -\frac{N-2^k}{2^{k+1}+2} + p_2^{-1}(N)
\end{align*}
miles into the desert with returning. From theorem \ref{jeepz}, it follows that the penetration depth for $4 \le N \le 6$ is as given. It remains to show that the term to the left of $p_2^{-1}(N)$ is at least $-1/(N-1)$. This follows from the fact that $(N-2^k)(N-1)$ is a quadratic function in $N$ which takes the value $2^{k+1} + 2$ for $N = -1$ and $N = 2^k + 2$.

Suppose next that $N \ge 2n + 2$. With $N$ bananas, the camel can penetrate
\begin{align*}
\frac{N-1-2n}{2n-1} + k &= \frac{N-1-2^k}{2^k-1} + k \\
&= -\frac{1}{2^k-1} + \frac{N-2^k}{2^k(2^k-1)} + p_2^{-1}\bigg(p_2\bigg(\frac{N-2^k}{2^k}\bigg)\bigg) + k \\
&= -\frac{2^k}{2^k(2^k-1)} + \frac{N-2^k}{2^k(2^k-1)} + p_2^{-1}\bigg(\frac{N}{2^k}\bigg) + k \\
&= -\frac{2^{k+1}-N}{2^k(2^k-1)} + p_2^{-1}(N) 
\end{align*}
miles into the desert with returning. From theorem \ref{jeepz}, it follows that the penetration depth for $6 \le N \le 8$ is as given. It remains to show that the the term to the left of $p_2^{-1}(N)$ is more than $-1/(N-1)$. This follows from the fact that $(2^{k+1}-N)(N-1)$ is a quadratic function in $N$ which takes the value $2^k(2^k-1)$ for $N = 2^k$ and $N = 2^k + 1$.
\end{proof}
 
If we want the camel somehow to penetrate $p_2^{-1}(N)$ miles into the desert with returning, then we can cheat. For instance, we can secretly give the camel an extra banana. A more suble way to cheat is not to count the final banana fuel in the camel's stomach as used fuel.

\begin{theorem}
In the round trip variant of the case $B=2$, the camel can penetrate exactly
$$
p_2^{-1}(N)
$$
miles into the desert if $N \ge 2$, provided we do not count the final banana fuel in the camel's stomach as used fuel.
\end{theorem}

\begin{proof}
The case $2 \le N \le 4$ follows from theorem \ref{roundtrip2}, so assume that $N \ge 4$.
Suppose that the camel eats it last banana on position $x$ in the return trip. After that meal, $P_2 = -\frac{B}2 \big(1 + \frac2B x\big) = -\frac{B}2 - x$. The camel finishes with $1-x$ units of banana fuel, so it can start with $N + 1 - x$ bananas. Therefore, $P_2 \le N + 1 - x - 1 - \frac{B}2 = N - \frac{B}2 - x$ at the beginning. So $P_2$ decreases at most $N$, just as before.

Take $k \in \N^{*}$ such that $2^k \le N \le 2^{k+1}$, and let $n = 2^{k-1}$. Let $N' = N + 2 - \frac{N}{2n}$. Then $2n + 1 \le N' \le 4n$. With $N'$ bananas, the camel can penetrate
\begin{align*}
\frac{N'-1-2n}{2n-1} + k &= \frac{N'-1-2^k}{2^k-1} + k
= \frac{N-\frac{N}{2^k}+1-2^k}{2^k-1} + k \\
&= \frac{N-2^k}{2^k} + k = p_2^{-1}\Big(\frac{N}{2^k}\Big) + k
= p_2^{-1}(N) 
\end{align*}
miles into the desert with returning. The camel returns with
$$
1 - \frac{N'-1-2n}{2n-1} = 1 - \frac{N-2^k}{2^k} = 2 - \frac{N}{2n}
$$
units of banana fuel in its stomach, so $N$ bananas are counted.
\end{proof}

\begin{lemma}
In the original camel-banana problem, we have the following.
\begin{compactenum}[(i)]

\item Suppose that $y \le z \le y + f$. If the camel is on position $y$ with $1$ banana next to him and $f$ bananas in its stomach, and there is $1$ banana on position $z$, then the camel can get as far as position $\frac13 y + \frac23 z + 2 + \frac13 f$.

\item Suppose that $x \le y \le x + f$. If the camel is on position $x$ with $1$ banana next to him and $f$ bananas in its stomach, and there are $2$ bananas on position $y$, then the camel can get as far as position $\frac13 x + \frac23 y + 2\frac13 + \frac13 f$.

\item Suppose that $x \le y \le x + \frac12$ and $n \in \N$ with $n \ge 2$. If the camel is on position $x$ with $4n - 2$ bananas and there are $2$ bananas on position $y$, then the camel can get $\frac12$ a mile farther than with $2n - 2$ bananas on position $x$ and $2$ bananas on position $\frac12 x + \frac12 y + \frac18$.

\end{compactenum}
\end{lemma}

\begin{proof}
\begin{compactenum}[(i)]

\item Let $w = \frac13 y + \frac23z + \frac13 f$. First, the camel carries a banana from position $y$ to position $w$, and fetches the banana on position $z$, returning to position $w$. This takes
$$
(w - y) + 2(w - z) = 3w - y - 2z = f
$$
miles. Finally, the camel walks from position $w$ to position $w + 2 = \frac13 y + \frac23z + 2 + \frac13 f$, using the $2$ remaining bananas.

\item Let $z = \frac12 x + \frac12 y + \frac12 f$. The camel carries a banana from position $x$ to $z$, and walks back to $y$. This takes
$$
(z - x) + (z - y) = 2z - x - y = f
$$
miles. After that, the camel can get as far as position 
$$
\tfrac13 y + \tfrac23 z + 2\tfrac13 
= \tfrac13 y + \tfrac13 x + \tfrac13 y + \tfrac13 f + 2\tfrac13
= \tfrac13 x + \tfrac23 y + 2\tfrac13 + \tfrac13 f
$$

\item Let $z = \frac12 x + \frac12 y + \frac12$. First, the camel eats $2n-2$ bananas on position $x$ and carries $2n - 2$ banana from position $x$ to position $x + \frac12$, returning to position $x$. This takes $4n - 4$ bananas on position $x$, so $2$ bananas on position $x$ remain. After that, the camel eats a banana on position $x$, carries a banana from position $x$ to position $z$, and walks back to position $y$. This takes
$$
(z - x) + (z - y) = 2z - x - y = 1
$$
miles. Next, the camel eats a banana on position $y$, carries a banana to position $z + \frac18$, fetches the banana on position $z$, and walks from position $z + \frac18$ back to position $x + \frac12$. This takes
\begin{align*}
\big(z + \tfrac18 - y\big) + 2\big(z + \tfrac18 - z\big) + \big(z + \tfrac18 - x - \tfrac12)
&= 4z + \tfrac12 - y - 2z - x - \tfrac12 \\
&= 2z - x - y = 1
\end{align*}
miles. After that, the camel is on position $x+\frac12$ with $2n-2$ bananas, and there are $2$ bananas on position $z + \frac18 = \frac12 x + \frac12 y + \frac18 + \frac12$.

\end{compactenum}
\end{proof}

\begin{theorem}
Let $k \in \N^{*}$ and $0 \le f \le 2$. With $2^k + f$ bananas, the camel of the original camel-banana problem can penetrate exactly
$$
\tfrac12 k + 1\tfrac56 + \frac{f-1}{3 \cdot 2^{k-1}}
$$
miles into the desert.
\end{theorem}

\begin{proof}
From corollary \ref{cor}, it follows that the camel of the original camel-banana problem cannot penetrate farther into the desert.

Suppose first that $k = 1$. Then the camel starts with $2 + f$ bananas, and can penetrate
$$
2 + \tfrac13 f = \tfrac12 + 1\tfrac56 - \tfrac13 + \tfrac13 f = \tfrac12 k + 1\tfrac56 + \frac{f-1}{3 \cdot 2^{k-1}}
$$
miles into the desert.

Suppose next that $k \ge 2$. Let $0 \le y \le \frac12$. Suppose that the camel is on position $0$ with $2^k - 2$ bananas, and there are $2$ bananas on position $y$. We show by induction that the camel can get as far as
$$
\tfrac12 k + 1\tfrac56 + \frac{4y - 1}{3 \cdot 2^{k-1}}
$$
miles into the desert. If $k = 2$, then the camel can get as far as 
$$
\tfrac23 y + 2\tfrac13 + \tfrac13 = 1 + 1\tfrac56 - \tfrac16 + \tfrac23 y = \tfrac12 k + 1\tfrac56 + \frac{4y - 1}{3 \cdot 2^{k-1}}
$$
miles into the desert. If $k \ge 3$, then the camel can get as far as  
$$
\tfrac12 + \tfrac12 (k-1) + 1\tfrac56 + \frac{4\big(\frac12y+\frac18\big)-1}{3 \cdot 2^{k-2}}
= \tfrac12 k + 1\tfrac56 + \frac{4y - 1}{3 \cdot 2^{k-1}}
$$
miles into the desert.

The camel starts with $2^k + f$ bananas. If $f \le 1$, then the camel first carries $2$ bananas to position $y = \frac14 f$, returning to position $0$. This takes $4y = f$ miles. If $f \ge 1$, then the camel first carries a banana to position $w = \frac12 f - \frac12$, returning to position $0$. This takes $2w = f - 1$ miles. Next, it carries a banana to position $y = \frac14 f$, fetches the banana on position $w$, and walks back from position $y$ to position $0$. This takes $y + 2(y-w) + y = 4y - 2w = f - (f-1) = 1$ mile. In both cases, carrying $2$ bananas to position $y = \frac14 f$ with returning to position $0$ takes $f$ miles. After that, the camel can penetrate
$$
\tfrac12 k + 1\tfrac56 + \frac{4y-1}{3 \cdot 2^{k-1}} = 
\tfrac12 k + 1\tfrac56 + \frac{f-1}{3 \cdot 2^{k-1}}
$$
miles into the desert.
\end{proof}

\end{document}